\documentclass[12pt,reqno]{amsart}

\usepackage{amssymb}

\catcode`\@=11

\thinlines
\newskip\Einheit \Einheit=0.6cm
\newcount\xcoord \newcount\ycoord
\newdimen\xdim \newdimen\ydim \newdimen\PfadD@cke \newdimen\Pfadd@cke
\def\PfadDicke#1{\PfadD@cke#1 \divide\PfadD@cke by2 \Pfadd@cke\PfadD@cke \multiply\PfadD@cke by2}
\long\def\LOOP#1\REPEAT{\def\BODY{#1}\ITERATE}
\def\ITERATE{\BODY \let\next\ITERATE \else\let\next\relax\fi \next}
\let\REPEAT=\fi
\def\Punkt{\hbox{\raise-2pt\hbox to0pt{\hss\scriptsize$\bullet$\hss}}}
\def\DuennPunkt(#1,#2){\unskip
  \raise#2 \Einheit\hbox to0pt{\hskip#1 \Einheit
          \raise-2.5pt\hbox to0pt{\hss\normalsize$\bullet$\hss}\hss}}
\def\NormalPunkt(#1,#2){\unskip
  \raise#2 \Einheit\hbox to0pt{\hskip#1 \Einheit
          \raise-3pt\hbox to0pt{\hss\large$\bullet$\hss}\hss}}
\def\DickPunkt(#1,#2){\unskip
  \raise#2 \Einheit\hbox to0pt{\hskip#1 \Einheit
          \raise-4pt\hbox to0pt{\hss\Large$\bullet$\hss}\hss}}
\def\Kreis(#1,#2){\unskip
  \raise#2 \Einheit\hbox to0pt{\hskip#1 \Einheit
          \raise-4pt\hbox to0pt{\hss\Large$\circ$\hss}\hss}}
\def\Diagonale(#1,#2)#3{\unskip\leavevmode
  \xcoord#1\relax \ycoord#2\relax
      \raise\ycoord \Einheit\hbox to0pt{\hskip\xcoord \Einheit
         \unitlength\Einheit
         \line(1,1){#3}\hss}}
\def\AntiDiagonale(#1,#2)#3{\unskip\leavevmode
  \xcoord#1\relax \ycoord#2\relax 
      \raise\ycoord \Einheit\hbox to0pt{\hskip\xcoord \Einheit
         \unitlength\Einheit
         \line(1,-1){#3}\hss}}
\def\Pfad(#1,#2),#3\endPfad{\unskip\leavevmode
  \xcoord#1 \ycoord#2 \thicklines\ZeichnePfad#3\endPfad\thinlines}
\def\ZeichnePfad#1{\ifx#1\endPfad\let\next\relax
  \else\let\next\ZeichnePfad
    \ifnum#1=1
      \raise\ycoord \Einheit\hbox to0pt{\hskip\xcoord \Einheit
         \vrule height\Pfadd@cke width1 \Einheit depth\Pfadd@cke\hss}%
      \advance\xcoord by 1
    \else\ifnum#1=2
      \raise\ycoord \Einheit\hbox to0pt{\hskip\xcoord \Einheit
        \hbox{\hskip-\PfadD@cke\vrule height1 \Einheit width\PfadD@cke depth0pt}\hss}%
      \advance\ycoord by 1
    \else\ifnum#1=3
      \raise\ycoord \Einheit\hbox to0pt{\hskip\xcoord \Einheit
         \unitlength\Einheit
         \line(1,1){1}\hss}
      \advance\xcoord by 1
      \advance\ycoord by 1
    \else\ifnum#1=4
      \raise\ycoord \Einheit\hbox to0pt{\hskip\xcoord \Einheit
         \unitlength\Einheit
         \line(1,-1){1}\hss}
      \advance\xcoord by 1
      \advance\ycoord by -1
    \else\ifnum#1=5
      \advance\xcoord by -1
      \raise\ycoord \Einheit\hbox to0pt{\hskip\xcoord \Einheit
         \vrule height\Pfadd@cke width1 \Einheit depth\Pfadd@cke\hss}%
    \else\ifnum#1=6
      \advance\ycoord by -1
      \raise\ycoord \Einheit\hbox to0pt{\hskip\xcoord \Einheit
        \hbox{\hskip-\PfadD@cke\vrule height1 \Einheit width\PfadD@cke depth0pt}\hss}%
    \else\ifnum#1=7
      \advance\xcoord by -1
      \advance\ycoord by -1
      \raise\ycoord \Einheit\hbox to0pt{\hskip\xcoord \Einheit
         \unitlength\Einheit
         \line(1,1){1}\hss}
    \else\ifnum#1=8
      \advance\xcoord by -1
      \advance\ycoord by +1
      \raise\ycoord \Einheit\hbox to0pt{\hskip\xcoord \Einheit
         \unitlength\Einheit
         \line(1,-1){1}\hss}
    \fi\fi\fi\fi
    \fi\fi\fi\fi
  \fi\next}
\def\hSSchritt{\leavevmode\raise-.4pt\hbox to0pt{\hss.\hss}\hskip.2\Einheit
  \raise-.4pt\hbox to0pt{\hss.\hss}\hskip.2\Einheit
  \raise-.4pt\hbox to0pt{\hss.\hss}\hskip.2\Einheit
  \raise-.4pt\hbox to0pt{\hss.\hss}\hskip.2\Einheit
  \raise-.4pt\hbox to0pt{\hss.\hss}\hskip.2\Einheit}
\def\vSSchritt{\vbox{\baselineskip.2\Einheit\lineskiplimit0pt
\hbox{.}\hbox{.}\hbox{.}\hbox{.}\hbox{.}}}
\def\DSSchritt{\leavevmode\raise-.4pt\hbox to0pt{%
  \hbox to0pt{\hss.\hss}\hskip.2\Einheit
  \raise.2\Einheit\hbox to0pt{\hss.\hss}\hskip.2\Einheit
  \raise.4\Einheit\hbox to0pt{\hss.\hss}\hskip.2\Einheit
  \raise.6\Einheit\hbox to0pt{\hss.\hss}\hskip.2\Einheit
  \raise.8\Einheit\hbox to0pt{\hss.\hss}\hss}}
\def\dSSchritt{\leavevmode\raise-.4pt\hbox to0pt{%
  \hbox to0pt{\hss.\hss}\hskip.2\Einheit
  \raise-.2\Einheit\hbox to0pt{\hss.\hss}\hskip.2\Einheit
  \raise-.4\Einheit\hbox to0pt{\hss.\hss}\hskip.2\Einheit
  \raise-.6\Einheit\hbox to0pt{\hss.\hss}\hskip.2\Einheit
  \raise-.8\Einheit\hbox to0pt{\hss.\hss}\hss}}
\def\SPfad(#1,#2),#3\endSPfad{\unskip\leavevmode
  \xcoord#1 \ycoord#2 \ZeichneSPfad#3\endSPfad}
\def\ZeichneSPfad#1{\ifx#1\endSPfad\let\next\relax
  \else\let\next\ZeichneSPfad
    \ifnum#1=1
      \raise\ycoord \Einheit\hbox to0pt{\hskip\xcoord \Einheit
         \hSSchritt\hss}%
      \advance\xcoord by 1
    \else\ifnum#1=2
      \raise\ycoord \Einheit\hbox to0pt{\hskip\xcoord \Einheit
        \hbox{\hskip-2pt \vSSchritt}\hss}%
      \advance\ycoord by 1
    \else\ifnum#1=3
      \raise\ycoord \Einheit\hbox to0pt{\hskip\xcoord \Einheit
         \DSSchritt\hss}
      \advance\xcoord by 1
      \advance\ycoord by 1
    \else\ifnum#1=4
      \raise\ycoord \Einheit\hbox to0pt{\hskip\xcoord \Einheit
         \dSSchritt\hss}
      \advance\xcoord by 1
      \advance\ycoord by -1
    \else\ifnum#1=5
      \advance\xcoord by -1
      \raise\ycoord \Einheit\hbox to0pt{\hskip\xcoord \Einheit
         \hSSchritt\hss}%
    \else\ifnum#1=6
      \advance\ycoord by -1
      \raise\ycoord \Einheit\hbox to0pt{\hskip\xcoord \Einheit
        \hbox{\hskip-2pt \vSSchritt}\hss}%
    \else\ifnum#1=7
      \advance\xcoord by -1
      \advance\ycoord by -1
      \raise\ycoord \Einheit\hbox to0pt{\hskip\xcoord \Einheit
         \DSSchritt\hss}
    \else\ifnum#1=8
      \advance\xcoord by -1
      \advance\ycoord by 1
      \raise\ycoord \Einheit\hbox to0pt{\hskip\xcoord \Einheit
         \dSSchritt\hss}
    \fi\fi\fi\fi
    \fi\fi\fi\fi
  \fi\next}
\def\Koordinatenachsen(#1,#2){\unskip
 \hbox to0pt{\hskip-.5pt\vrule height#2 \Einheit width.5pt depth1 \Einheit}%
 \hbox to0pt{\hskip-1 \Einheit \xcoord#1 \advance\xcoord by1
    \vrule height0.25pt width\xcoord \Einheit depth0.25pt\hss}}
\def\Koordinatenachsen(#1,#2)(#3,#4){\unskip
 \hbox to0pt{\hskip-.5pt \ycoord-#4 \advance\ycoord by1
    \vrule height#2 \Einheit width.5pt depth\ycoord \Einheit}%
 \hbox to0pt{\hskip-1 \Einheit \hskip#3\Einheit 
    \xcoord#1 \advance\xcoord by1 \advance\xcoord by-#3 
    \vrule height0.25pt width\xcoord \Einheit depth0.25pt\hss}}
\def\Gitter(#1,#2){\unskip \xcoord0 \ycoord0 \leavevmode
  \LOOP\ifnum\ycoord<#2
    \loop\ifnum\xcoord<#1
      \raise\ycoord \Einheit\hbox to0pt{\hskip\xcoord \Einheit\Punkt\hss}%
      \advance\xcoord by1
    \repeat
    \xcoord0
    \advance\ycoord by1
  \REPEAT}
\def\Gitter(#1,#2)(#3,#4){\unskip \xcoord#3 \ycoord#4 \leavevmode
  \LOOP\ifnum\ycoord<#2
    \loop\ifnum\xcoord<#1
      \raise\ycoord \Einheit\hbox to0pt{\hskip\xcoord \Einheit\Punkt\hss}%
      \advance\xcoord by1
    \repeat
    \xcoord#3
    \advance\ycoord by1
  \REPEAT}
\def\Label#1#2(#3,#4){\unskip \xdim#3 \Einheit \ydim#4 \Einheit
  \def\lo{\advance\xdim by-.5 \Einheit \advance\ydim by.5 \Einheit}%
  \def\llo{\advance\xdim by-.25cm \advance\ydim by.5 \Einheit}%
  \def\loo{\advance\xdim by-.5 \Einheit \advance\ydim by.25cm}%
  \def\o{\advance\ydim by.25cm}%
  \def\ro{\advance\xdim by.5 \Einheit \advance\ydim by.5 \Einheit}%
  \def\rro{\advance\xdim by.25cm \advance\ydim by.5 \Einheit}%
  \def\roo{\advance\xdim by.5 \Einheit \advance\ydim by.25cm}%
  \def\l{\advance\xdim by-.30cm}%
  \def\r{\advance\xdim by.30cm}%
  \def\lu{\advance\xdim by-.5 \Einheit \advance\ydim by-.6 \Einheit}%
  \def\llu{\advance\xdim by-.25cm \advance\ydim by-.6 \Einheit}%
  \def\luu{\advance\xdim by-.5 \Einheit \advance\ydim by-.30cm}%
  \def\u{\advance\ydim by-.30cm}%
  \def\ru{\advance\xdim by.5 \Einheit \advance\ydim by-.6 \Einheit}%
  \def\rru{\advance\xdim by.25cm \advance\ydim by-.6 \Einheit}%
  \def\ruu{\advance\xdim by.5 \Einheit \advance\ydim by-.30cm}%
  #1\raise\ydim\hbox to0pt{\hskip\xdim
     \vbox to0pt{\vss\hbox to0pt{\hss$#2$\hss}\vss}\hss}%
}
\catcode`\@=12

\numberwithin{equation}{section}

\newtheorem{Theorem}{Theorem}

\newtheorem{Proposition}[Theorem]{Proposition}

\newtheorem{Corollary}[Theorem]{Corollary}
\newtheorem*{Fact}{Fact}
\newtheorem{Conjecture}[Theorem]{Conjecture}

\theoremstyle{definition}

\theoremstyle{remark}

\newtheorem*{Remark}{Remark}

\def\al{\alpha}
\def\be{\beta}

\def\la{\lambda}

\def\si{\sigma}
\def\ta{\tau}

\def\Si{\Sigma}

\def\Om{\Omega}

\def\today{\ifcase\month\or
 January\or February\or March\or April\or May\or June\or
 July\or August\or September\or October\or November\or December\fi
 \space\number\day, \number\year}

\def\({\left(}
\def\){\right)}
\def\[{\left[}
\def\]{\right]}

\def\3{\ss}

\begin{document}

\newbox\Adr
\setbox\Adr\vbox{
\centerline{\sc Ting Guo$^{\ast}$, Christian Krattenthaler$^{\dagger}$,
  Yi Zhang$^{\ddagger}$ } 
\vskip18pt
\centerline{\footnotesize$^\ast$College of Mathematics and Statistics, Hunan Normal University}
\centerline{\footnotesize$^\dagger$Fakult\"at f\"ur Mathematik, 
Universit\"at Wien}
\centerline{\footnotesize$^\ddagger$Johann Radon Institute for Computational
  and Applied Mathematics (RICAM),}
\centerline{\footnotesize Austrian Academy of Sciences}
}

\title
{On (shape-)Wilf-equivalence for words}

\author[T. Guo, C. Krattenthaler and Y. Zhang]{\box\Adr}

\address{$^\ast$Key Laboratory of High Performance Computing and Stochastic Information Processing
(HPCSIP, Ministry of Education of China), College of Mathematics and Statistics, Hunan Normal University, Changsha, Hunan 410081, P. R. China}
\email{guoting@hunnu.edu.cn}

\address{$^\dagger$Fakult\"at f\"ur Mathematik, Universit\"at Wien,
Oskar-Morgenstern-Platz 1, 1090 Vienna, Austria.}
\urladdr{http://www.mat.univie.ac.at/~kratt}

\address{$^\ddagger$Johann Radon Institute for Computational and Applied Mathematics (RICAM), Austrian Academy of Sciences, Altenberger Str. 69, Linz, Austria}

\thanks{
$^\ast$The work of the first author was done at the University of
  Vienna while she was visiting the Fakult\"at f\"ur Mathematik,
  supported by the China Scholarship Council. 
 \\ 
\indent$^\dagger$The work of the second author is partially supported by the
  Austrian Science Fund FWF, grant SFB F50 
  (Special Research Programme ``Algorithmic and Enumerative
  Combinatorics"). {\it WWW}: {\tt
    http://www.mat.univie.ac.at/\lower0.5ex\hbox{\~{}}kratt}.\\ 
\indent$^\ddagger$The third author is supported by the Austrian Science Fund
  FWF, grant P29467-N32. {\it Email}: {\tt zhangy@amss.ac.cn}}

\subjclass[2010]{Primary 05A15;
 Secondary 05A17 05A19 05E10}
\keywords{Words, pattern avoidance, Ferrers diagrams, Ferrers shapes, 
growth diagrams, 
Wilf-equivalence, shape-Wilf-equivalence}

\begin{abstract}
Stankova and West showed that for any non-negative integer~$s$
and any permutation $\gamma$
of $\{4,5,\dots,s+3\}$ there are as many permutations that avoid 
$231\gamma$ as there are that avoid $312\gamma$.
We extend this result to the setting of words.
\end{abstract}

\maketitle

\section{Introduction} 
\label{sec:Intro}

Let $\pi=\pi_1\pi_2\cdots \pi_n$ be a permutation of $\{1,2,\dots,n\}$
and $\si=\si_1\si_2\cdots\si_r$ be a permutation of $\{1,2,\dots,r\}$,
$r\le n$. We say that the permutation $\pi$ {\it contains the pattern}
$\si$, if there 
are indices $1\le i_1<i_2<\dots<i_r\le n$ such that 
$\pi_{i_1}\pi_{i_2}\cdots \pi_{i_r}$ is in the same relative order as
$\si_1\si_2\cdots\si_r$. Otherwise, $\pi$ is said to {\it avoid the pattern}
$\si$, or, alternatively, we say that $\pi$ is {\it $\si$-avoiding}.
As usual, we write $S_n$ for the set of all permutations of $\{1,2,\dots,n\}$,
and $S_n(\si)$ for the set of permutations in $S_n$ that avoid $\si$. 

The enumeration of permutations which avoid certain patterns has been
a flourishing research subject since the seminal
article \cite{SiScAA} of Simion and Schmidt, where
this research subject was ``defined."
The reader is referred to \cite{KitaAA} and \cite[Chapters~4
  and~5]{BonaAZ} for 
in-depth accounts of the enumeration of pattern avoiding permutations.

We define the {\it direct sum} $\si\oplus\ta$ of two
permutations $\si=\si_1\si_2\cdots\si_r\in S_r$ and 
$\ta=\ta_1\ta_2\cdots\ta_s\in S_s$
as 
$$\si\oplus\ta=\si_1\si_2\cdots\si_r(\ta_1+r)(\ta_2+r)\cdots(\ta_s+r).$$

The starting point for the work on this paper was an email of
Doron Zeilberger \cite{ZeilAZ} to the second author
saying\footnote{Notation is slightly adapted in order to conform with
the notation of our paper.}:

{\bigskip\parindent0cm
\raggedright\leftskip1cm\rightskip1cm
\tt
According to http://en.wikipedia.org/wiki/Permutation\underline{ }pattern,

Backelin, West \& Xin (2007) proved that for any permutation\newline 
beta and any positive integer k, the permutations\newline 
12..k "+" beta
and
k....21 "+" beta
are Wilf equivalent,

\medskip
and in 2002 Stankova and West   proved that 
231 "+" beta 
and \newline
312 "+" beta are Wilf equivalent.

\medskip
According to Vince Vatter and Jonathan Bloom you have nice\newline 
proofs at least of the 
first result, in "Growth Diagrams, ...\newline Ferrers Shapes".

\medskip
My main question is whether it is true, and whether there\newline
exists a proof, of the above
two results generalized to words\newline 
with $a_1$ 1's, $a_2$ 2's, ..., $a_m$ $m$'s,
where the original case is\newline $a_1=...=a_m=1$. 
\bigskip

}

\noindent
Here, two patterns $\si$ and $\tau$ are said to be {\it Wilf-equivalent},
denoted by $\si\sim\tau$, if $\vert S_n(\si)\vert=\vert
S_n(\tau)\vert$ for all positive integers~$n$. What Zeilberger
refers to are the following two results.

\begin{Theorem}[{\cite[Theorem~2.1]{BaWXAA}}] 
\label{thm:BackWeXi}
For all positive integers $k$, all non-negative integers $s$, 
and all patterns $\be\in S_s$,
the patterns $12\cdots k\oplus\be$ and $k\cdots21\oplus\be$ are
Wilf-equivalent. 
\end{Theorem}

\begin{Theorem}[{\cite[Theorem~1]{StWeAA}}] 
\label{thm:StWe}
For all non-negative integers $s$ and all patterns $\be\in S_s$, the
patterns $231\oplus \be$ and $312\oplus\be$ are Wilf-equivalent.  
\end{Theorem}

In fact, these two theorems explain all instances of 
Wilf-equivalence (of single patterns) that are known up to this date,
except for $1423\sim 3142$ (and equivalent Wilf-equivalence relations
arising from reversal and/or complementation of patterns;
cf.\ \cite[pp.~134/135]{BaWXAA}).

\medskip
Now, clearly, the notion of pattern avoidance can be straightforwardly
adapted to {\it words} over the alphabet of positive integers.
Given a word $w=w_1w_2\cdots w_n$ 
and a word $x=x_1x_2\cdots x_r$ with the letters $w_i$ and $x_i$ taken
from the positive integers, we say that the word $w$ 
{\it contains the pattern} $x$ if there 
are indices $1\le i_1<i_2<\dots<i_r\le n$ such that 
$w_{i_1}w_{i_2}\cdots w_{i_r}$ is in the same relative order as
$x_1x_2\cdots x_r$. Otherwise, $w$ is said to {\it avoid the pattern}
$x$. Since in words there can be equal letters, by ``same relative
order" we mean that, if there is an equality $x_s=x_t$ then there
must also hold the equality $w_{i_s}=w_{i_t}$. Thus, Zeilberger asks
whether Theorems~\ref{thm:BackWeXi} and~\ref{thm:StWe} extend to
words with a given number of 1's, 2's, \dots, $m$'s, the theorems themselves
then being the special case where one considers words with exactly
one~1, exactly one~2, \dots, and exactly one~$m$.

As it turns out, the first question had already been answered by
Jel\'\i nek and Mansour in \cite{JeMaAA}.
The purpose of this note is to answer the second question.

Before we describe the corresponding results, it is however helpful to recall that the
key concept behind the proofs of Theorems~\ref{thm:BackWeXi} 
and~\ref{thm:StWe} is {\it shape-Wilf-equivalence}, a concept introduced in
\cite{BaWXAA} (already non-explicitly present in \cite{BaWeAA};
see
Section~\ref{sec:def} for its definition). This is a stronger notion
than Wilf-equivalence. In particular, if two permutations $\si$ and
$\tau$ are {\it shape-Wilf-equivalent}, then they are also
{\it Wilf-equivalent}. The significance of shape-Wilf-equivalence
is explained by the following proposition.

\begin{Proposition}[{\cite[Prop.~2.3]{BaWXAA}}] 
\label{prop:sW}
If the permutations $\beta$ and $\gamma$ are
  shape-Wilf-equivalent, then for every permutation $\delta$,
  $\beta\oplus\delta$ and $\gamma\oplus\delta$ are also
  shape-Wilf-equivalent. 
\end{Proposition} 

In plain terms, given two shape-Wilf-equivalent permutations, 
by forming direct sums we can
obtain a whole infinite family of pairs of shape-Wilf-equivalent
permutations, which are automatically also pairs of Wilf-equivalent
permutations.
In particular, Theorems~\ref{thm:BackWeXi} and~\ref{thm:StWe} follow from
the following two results. 

\begin{Theorem}[{\cite[Theorem~2.1]{BaWXAA}}] 
\label{thm:I-J}
For every positive integer
  $k$, the permutations $12\cdots k$ and $k\cdots21$ are
  shape-Wilf-equivalent. 
\end{Theorem}

\begin{Theorem}[{\cite[Theorem~1]{StWeAA}}] 
\label{thm:231-312}
The permutations $231$ and
  $312$ are shape-Wilf-equi\-va\-lent.  
\end{Theorem}

In view of Proposition~\ref{prop:sW}, one sees that in
Theorems~\ref{thm:BackWeXi} and~\ref{thm:StWe} ``Wilf-equivalent"
may actually be replaced by ``shape-Wilf-equivalent."

\medskip
In \cite{JeMaAA}, Jel\'\i nek and Mansour
show that
Theorem~\ref{thm:BackWeXi} extends directly to the word setting.
The key to prove this is again shape-Wilf-equivalence, here
extended to words (see Section~\ref{sec:def} for the definition;
for the sake of the current discussion, it is only relevant to
know that {\em shape-Wilf-equivalence} of $x$ and $y$ implies
{\em strong Wilf-equivalence} of~$x$ and~$y$, the latter meaning
by definition that for all sequences $a_1,a_2,\dots,a_m$ of
positive integers
there are as many words with $a_i$ letters~$i$, $i=1,2,\dots,m$, 
and avoiding~$x$
as there are words with $a_i$ letters~$i$, $i=1,2,\dots,m$,  
and avoiding~$y$). In terms of this notion, Zeilberger's questions
translate into the questions whether $12\cdots k\oplus\beta$
and $k\cdots21\oplus\beta$ are strongly Wilf-equivalent, and
whether $231\oplus\beta$ and $312\oplus\beta$ are strongly Wilf-equivalent,
for any pattern~$\beta$.

\begin{Proposition}[{\cite[Lemma~2.1]{JeMaAA}}] \label{prop:5}
Let $x$ and $y$ be two words with letters from $\{1,2,\dots,k\}$ and
$w$ another word.
If $x$ and $y$ are shape-Wilf-equivalent for words then so
are $x\oplus w$ and $y\oplus w$.
\end{Proposition} 

For the definition of the direct sum of words see 
the end of Section~\ref{sec:def}.

\begin{Theorem}[{\cite[proof of Theorem~13, Eq.~(4.6)]{KratCE}}] \label{thm:6}
The words $12\cdots k$ and $k\cdots21$ are shape-Wilf-equivalent
for words.
\end{Theorem}

If these two results are combined, the answer to Zeilberger's first question
is obtained.

\begin{Corollary}[{\cite[Fact~2.2]{JeMaAA}}] \label{cor:7}
For all words $w$, the words $12\cdots k\oplus w$ 
and $k\cdots21\oplus w$ are shape-Wilf-equivalent for words.
\end{Corollary}

The situation is different for Theorems~\ref{thm:StWe} and~\ref{thm:231-312}.
The counterexamples in Section~\ref{sec:counter} demonstrate that
Theorem~\ref{thm:231-312} does not straightforwardly extend to the
word setting, that is, 
that $231$ and $312$ are {\it not shape}-Wilf-equivalent for words.
Moreover, our computer calculations strongly indicate that
$231\oplus\beta$ and $312\oplus\beta$ are never Wilf-equivalent
for words
when $\beta$ is a non-empty permutation, regardless of the
exact notion of ``Wilf-equivalence" that we consider (see
Conjecture~\ref{conj:2}).
On the other hand, it seems more difficult to avoid $312$ than $231$.
This vague statement is made more precise in Conjecture~\ref{conj:1}.

Nevertheless, weak versions of Theorems~\ref{thm:StWe} and \ref{thm:231-312}
hold true. 
If we do not insist on the ``strict" patterns
$231$ and $312$ but instead, say, interpret an occurrence of the patterns
$231$ and $312$ to mean that the `$2$' and the `$3$' could also ``be equal", then it is possible
to extend Theorems~\ref{thm:StWe} and~\ref{thm:231-312}. 
Theorem~\ref{lem:11} in Section~\ref{sec:231-312} 
says that the {\it sets of patterns} $\{231,221\}$ and $\{312,212\}$ are
shape-Wilf-equivalent for words, and 
Theorem~\ref{lem:12} in the same section 
says that $\{231,121\}$ and $\{312,211\}$ are
shape-Wilf-equivalent for words. 

Furthermore, although shape-Wilf-equivalence of $231$ and
$312$ for words is not true,
the (strong) Wilf-equivalence of $231$ and $312$ for words
(and, actually, of all permutation patterns of length~$3$)
is again known and follows as a special case from another result
of Jel\'\i nek and Mansour.

\begin{Theorem}[{\cite[Lemma~2.4]{JeMaAA}}]
For any $k$, all the patterns that consist of a single letter `$1$', 
a single letter
`$3$' and $k - 2$ letters `$2$' are strongly Wilf-equivalent.
\end{Theorem}

Before describing our results, 
we need to collect definitions and notation in the next section.

\section{Definitions and notation}
\label{sec:def}

We recall from the introduction that $S_n(\si)$ denotes
the set of all permutations of $\{1,2,\dots,n\}$ that avoid~$\si$.
We choose a similar notation in the word setting.
We write $W^{(a_1,a_2,\dots)}(x)$ for the set of all
words consisting of $a_i$ letters~$i$, $i=1,2,\dots$, and avoiding~$x$.
More generally, given a set $\Om$ of words, we write 
$W^{(a_1,a_2,\dots)}(\Om)$ for the set of all
words consisting of $a_i$ letters~$i$, $i=1,2,\dots$, and avoiding
all~$x\in\Om$. (For convenience, the set braces may be omitted in this
notation.)
Occasionally, we shall use the symbol $W_{n,m}(x)$ to
denote the set of {\it all\/} words of length~$n$ that 
use letters from $\{1,2,\dots,m\}$ and avoid the word~$x$.

Next we reformulate pattern avoidance of permutations
in terms of matrices. Clearly, any permutation $\pi=\pi_1\pi_2\cdots
\pi_n$ can be represented by the corresponding {\it permutation
  matrix} $M(\pi)$, which for us is the
$n\times n$ matrix which contains a $1$ in
column~$j$ and row $\pi_j$, $j=1,2,\dots,n$, and all other entries
are zero. When we display a
matrix, our convention is that rows are ordered from bottom to
top, so that row~1 is the lowest row and row~$n$ the top-most.
For example, the left of Figure~\ref{fig:0} shows the matrix
corresponding to the permutation $352164$.
Alternatively, we may represent an $n\times n$ permutation matrix as a
$0$-$1$-filling of a square arrangement of cells, with $n$ cells in
each row and in each column, such that each row and each column
contain exactly one~$1$ and otherwise $0$'s. In the interest of
better readability, instead of $1$'s we write $X$'s, and we suppress
the $0$'s, see the right of Figure~\ref{fig:0}.

\begin{figure}[h]
$$M(352164)=\begin{pmatrix} 
0&0&0&0&1&0\\
0&1&0&0&0&0\\
0&0&0&0&0&1\\
1&0&0&0&0&0\\
0&0&1&0&0&0\\
0&0&0&1&0&0
\end{pmatrix}
\hbox{\hskip3cm}
\Einheit.2cm
\PfadDicke{.5pt}
\Pfad(0,9),111111111111111111\endPfad
\Pfad(0,6),111111111111111111\endPfad
\Pfad(0,3),111111111111111111\endPfad
\Pfad(0,0),111111111111111111\endPfad
\Pfad(0,-3),111111111111111111\endPfad
\Pfad(0,-6),111111111111111111\endPfad
\Pfad(0,-9),111111111111111111\endPfad
\Pfad(18,-9),222222222222222222\endPfad
\Pfad(15,-9),222222222222222222\endPfad
\Pfad(12,-9),222222222222222222\endPfad
\Pfad(9,-9),222222222222222222\endPfad
\Pfad(6,-9),222222222222222222\endPfad
\Pfad(3,-9),222222222222222222\endPfad
\Pfad(0,-9),222222222222222222\endPfad
\Label\ro{\text {\small$X$}}(1,-2)
\Label\ro{\text {\small$X$}}(4,4)
\Label\ro{\text {\small$X$}}(7,-5)
\Label\ro{\text {\small$X$}}(10,-8)
\Label\ro{\text {\small$X$}}(13,7)
\Label\ro{\text {\small$X$}}(16,1)
\hskip3cm
$$
\caption{Matrix representation of $352164$}
\label{fig:0}
\end{figure}

Clearly, a permutation $\pi$ avoids the permutation~$\si$ if and only
if $M(\pi)$ does not contain $M(\si)$ as a submatrix.
For example, our permutation $352164$ avoids $1324$ but does not
avoid $231$.


Pattern avoidance in words can also be formulated in terms of
matrices. 
We may represent a word $w=w_1w_2\cdots $ 
with $a_1$ letters~1, $a_2$ letters~2, $\dots$,
$a_m$ letters~$m$ as a $0$-$1$-filling of 
an $m\times(a_1+a_2+\cdots+a_m)$ rectangle, 
by placing a $1$ into the $j$-th column and $w_j$-th row,
$j=1,2,\dots,a_1+a_2+\cdots+a_m$, all other entries being $0$.
We denote this $0$-$1$-filling by $M(w)$.
For example, the word $213314242$ is represented by the filling of
Figure~\ref{fig:1}.

\begin{figure}[h]
$$
\Einheit.2cm
\PfadDicke{.5pt}
\Pfad(0,12),111111111111111111111111111\endPfad
\Pfad(0,9),111111111111111111111111111\endPfad
\Pfad(0,6),111111111111111111111111111\endPfad
\Pfad(0,3),111111111111111111111111111\endPfad
\Pfad(0,0),111111111111111111111111111\endPfad
\Pfad(27,0),222222222222\endPfad
\Pfad(24,0),222222222222\endPfad
\Pfad(21,0),222222222222\endPfad
\Pfad(18,0),222222222222\endPfad
\Pfad(15,0),222222222222\endPfad
\Pfad(12,0),222222222222\endPfad
\Pfad(9,0),222222222222\endPfad
\Pfad(6,0),222222222222\endPfad
\Pfad(3,0),222222222222\endPfad
\Pfad(0,0),222222222222\endPfad
\Label\ro{\text {\small$X$}}(1,4)
\Label\ro{\text {\small$X$}}(4,1)
\Label\ro{\text {\small$X$}}(7,7)
\Label\ro{\text {\small$X$}}(10,7)
\Label\ro{\text {\small$X$}}(13,1)
\Label\ro{\text {\small$X$}}(16,10)
\Label\ro{\text {\small$X$}}(19,4)
\Label\ro{\text {\small$X$}}(22,10)
\Label\ro{\text {\small$X$}}(25,4)
\hskip6cm
$$
\caption{The $0$-$1$-filling corresponding to $213314242$}
\label{fig:1}
\end{figure}

Similarly as for permutations, 
a word $w$ avoids the word~$x$ if and only
if $M(w)$ does not contain $M(x)$ as a submatrix.
For example, our word $213314242$ avoids $3112$ but does not
avoid $123$. 

We are now in the position to define (strong) Wilf-equivalence and
shape-Wilf-equivalence for words.
Let $\la$ be a Ferrers shape, which is a left-justified arrangement of
cells with the property that the row-lengths are non-increasing from
bottom to top. (That is, we use the French convention when we
represent Ferrers shapes.) As usual, we encode $\la$ in terms of
$(\la_1,\la_2,\dots)$ where $\la_i$ is the length of the $i$-th row
of $\la$ (counted from bottom to top).
We write $W^{(a_1,a_2,\dots)}_\la(x)$ for the set of all
$0$-$1$-fillings with {\it exactly one}~$1$ in each column, with
$a_i$ $1$'s in row~$i$, $i=1,2,\dots$, and which avoid~$x$. 
Here, for a $0$-$1$-filling~$F$ of the Ferrers shape~$\la$ 
to avoid~$x$ means that there do not exist rows $r_1,r_2,\dots$
and columns $c_1,c_2,\dots$ such that the entries of~$F$ corresponding
to these rows and columns form a matrix which is identical with $M(x)$.
Phrased differently, the important point is that the {\it complete}
matrix $M(x)$ is found as a submatrix in the filling.
Figure~\ref{fig:2} shows a $0$-$1$-filling of
shape $(10,10,10,7,4,4)$ with exactly one~$1$ in each column.
It avoids for example the pattern $4312$. 

\begin{figure}[h]
$$
\Einheit.2cm
\PfadDicke{.5pt}
\Pfad(0,18),111111111111\endPfad
\Pfad(0,15),111111111111\endPfad
\Pfad(0,12),111111111111111111111\endPfad
\Pfad(0,9),111111111111111111111111111111\endPfad
\Pfad(0,6),111111111111111111111111111111\endPfad
\Pfad(0,3),111111111111111111111111111111\endPfad
\Pfad(0,0),111111111111111111111111111111\endPfad
\Pfad(30,0),222222222\endPfad
\Pfad(27,0),222222222\endPfad
\Pfad(24,0),222222222\endPfad
\Pfad(21,0),222222222222\endPfad
\Pfad(18,0),222222222222\endPfad
\Pfad(15,0),222222222222\endPfad
\Pfad(12,0),222222222222222222\endPfad
\Pfad(9,0),222222222222222222\endPfad
\Pfad(6,0),222222222222222222\endPfad
\Pfad(3,0),222222222222222222\endPfad
\Pfad(0,0),222222222222222222\endPfad
\Label\ro{\text {\small$X$}}(1,13)
\Label\ro{\text {\small$X$}}(4,7)
\Label\ro{\text {\small$X$}}(7,1)
\Label\ro{\text {\small$X$}}(10,16)
\Label\ro{\text {\small$X$}}(13,10)
\Label\ro{\text {\small$X$}}(16,4)
\Label\ro{\text {\small$X$}}(19,7)
\Label\ro{\text {\small$X$}}(22,1)
\Label\ro{\text {\small$X$}}(25,7)
\Label\ro{\text {\small$X$}}(28,4)
\hskip7cm
$$
\caption{A $0$-$1$-filling of shape $(10,10,10,7,4,4)$}
\label{fig:2}
\end{figure}

More generally, given a set~$\Om$ of words,
we write $W^{(a_1,a_2,\dots)}_\la(\Om)$ for the analogous set 
of $0$-$1$-fillings all of which avoid all $x\in\Om$. 

We call two words $x$ and $y$ {\it strongly Wilf-equivalent for words} if
\begin{equation} \label{eq:W-equiv} 
\vert W^{(a_1,a_2,\dots)}_{R(\mathbf a)}(x)\vert
=\vert W^{(a_1,a_2,\dots)}_{R(\mathbf a)}(y)\vert
\end{equation}
for all finite sequences $\mathbf a=a_1,a_2,\dots$ of positive integers, 
where $R(\mathbf a)$
denotes the $m\times(a_1+a_2+\cdots)$ rectangle.
Equivalently, in view of the earlier explained representations of words in 
terms of fillings, the words $x$ and $y$ are strongly Wilf-equivalent for words if and
only if the number of words consisting of exactly $a_i$ letters $i$,
$i=1,2,\dots$, and avoiding~$x$ is the same as the number of
words consisting of exactly $a_i$ letters $i$,
$i=1,2,\dots$, and avoiding~$y$, for all finite sequences
$a_1,a_2,\dots$ of positive integers. On the other hand,
``ordinary"
Wilf-equivalence of words~$x$ and~$y$ means that
$\vert W_{n,m}(x)\vert =\vert W_{n,m}(y)\vert$ for all~$n$ and~$m$.

We call two words $x$ and $y$ {\it shape-Wilf-equivalent for words} if
\begin{equation} \label{eq:sW-equiv} 
\vert W^{(a_1,a_2,\dots)}_{\la}(x)\vert
=\vert W^{(a_1,a_2,\dots)}_{\la}(y)\vert
\end{equation}
for all finite sequences $\mathbf a=a_1,a_2,\dots$ of positive integers
and all shapes $\la$.
We extend the notions of strong Wilf-equivalence and shape-Wilf-equivalence
to {\it sets} of words, say $\Om$ and $\Si$, by requiring that
\eqref{eq:W-equiv} respectively \eqref{eq:sW-equiv} hold with $x$
replaced by~$\Om$ and $y$ replaced by~$\Si$ for all finite sequences
$a_1,a_2,\dots$ of positive integers and shapes~$\la$.

If we restrict the above definitions to $a_1=a_2=\dots=1$, then
they reduce to Wilf-equivalence and shape-Wilf-equivalence for
permutations.

Finally, given a word $x=x_1x_2\cdots x_r$ with letters 
from $\{1,2,\dots,m\}$ and
another word $y=y_1y_2\cdots y_s$, 
we define the {\it direct sum} $x\oplus y$ by
$$x\oplus y=x_1x_2\cdots x_r(y_1+m)(y_2+m)\cdots(y_s+m).$$

\section{Counterexamples} 
\label{sec:counter}

We have done calculations by computer for the word patterns $231$ and
$312$. Consider the shapes
$\la_1=(5,5,4)$, $\la_2=(5,5,5,4)$, and $\la_3=(6,6,6,4)$. 
Table~\ref{table:1} presents the results of our computations for all
possible sequences $\mathbf a=a_1,a_2,\dots$ of positive integers. 
(Recall that, since
our $0$-$1$-fillings always have exactly one~$1$ in each column, the sum
of the $a_i$'s must equal the length of the longest row of the shape.
Thus, for $\la_1$ and $\la_2$ the sum of the $a_i$'s must be~$5$,
whereas for $\la_3$ the sum of the $a_i$'s must be~$6$.)
 
\begin{table}[h]
\footnotesize
\begin{tabular}{llllllllllll}
\hline
\hline
$\mathbf a$ & $(221)$ & $(212)$ & $(122)$ & $(311)$ & $(131)$ & $(113)$ \\
\hline
$\vert W_{\la_1}^{\mathbf a}(231)\vert$ & $18$ & $15$ & $15$ & $13$ & $13$ & $8$\\
\hline
$\vert W_{\la_1}^{\mathbf a}(312)\vert$ & $18$ & $15$ & $15$ & $13$ & $13$ & $8$\\
\hline
\\
\hline
\hline
${\mathbf a}$ & $(2111)$ & $(1211)$ & $(1121)$ & $(1112)$   \\
\hline
$\vert W_{\la_2}^{\mathbf a}(231)\vert$ & $25$ & $26$ & $25$ & $21$  \\
\hline
$\vert W_{\la_2}^{\mathbf a}(312)\vert$ & $25$ & $25$ & $26$ & $21$ \\
\hline
\\
\hline
\hline
${\mathbf a}$ & $(1113)$ & $(3111)$ & $(1311)$ & $(1131)$ & $(2112)$ &
$(1212)$ &  $(1122)$ & $(2211)$ & $(2121)$ & $(1221)$ \\ 
\hline
$\vert W_{\la_3}^{\mathbf a}(231)\vert$ & $20$ & $42$ & $40$ & $42$ & $42$ &
$39$ & $42$ & $52$ & $54$ & $52$ \\ 
\hline
$\vert W_{\la_3}^{\mathbf a}(312)\vert$ & $20$ & $42$ & $42$ & $42$ & $42$ &
$42$ & $39$ & $54$ & $53$ & $53$ \\ 
\hline
\end{tabular}
\vskip10pt
\caption{}
\label{table:1}
\end{table}

We make the following observations. 

\begin{enumerate}
\item For $\la_1$ and arbitrary sequences $\mathbf a=a_1,a_2,a_3$
of positive integers,
the sets $W_{\la_1}^{(a_1,a_2,a_3)}(231)$ and
$W_{\la_1}^{(a_1,a_2,a_3)}(312)$ have the same cardinality. 
\item For $\la_2$, although the total numbers of $0$-$1$-fillings
with exactly one~$1$ in each column and at least one~$1$ in each
row
that avoid $231$ respectively $312$ are the same,
there exist positive integers $a_1,a_2,a_3,a_4$ for which
$\vert W_{\la_2}^{(a_1,a_2,a_3,a_4)}(231)\vert$ is different from $\vert
W_{\la_2}^{(a_1,a_2,a_3,a_4)}(312)\vert$. 
\item For $\la_3$, even the total numbers do not match.
Instead, as one lets $a_1,a_2,a_3,a_4$ grow,
one finds more and more sequences of positive integers $a_1,a_2,a_3,a_4$ 
for which $\vert W_{\la_3}^{(a_1,a_2,a_3,a_4)}(231)\vert$ is different
from $\vert
W_{\la_3}^{(a_1,a_2,a_3,a_4)}(312)\vert$. 
\end{enumerate}

This establishes the following fact.

\begin{Fact} \label{thm:9} 
The patterns $231$ and $312$ are not
shape-Wilf-equivalent for words. 
\end{Fact}

Consequently, Proposition~\ref{prop:5} cannot be applied to conclude
that $231\oplus\beta$ and $312\oplus\beta$ are (strongly) Wilf-equivalent
for words for arbitrary patterns~$\beta$.
However, the latter might still be true, at least for some patterns~$\beta$.
In order to clarify that point as well, we made more computer
calculations.

First, we looked at the case where $\beta=1$, that is,
at words (and, more generally, $0$-$1$-fillings) avoiding $2314$ 
respectively $3124$.
We found that 
$$\vert
W_{7,5}(2314)\vert=67853\neq67854=\vert W_{7,5}(3124)\vert,$$ 
and,
with $\la_4=(7,7,7,7,7)$, the finer enumerations
$$\vert W_{\la_4}^{(1,2,1,2,1)}(2314)\vert=908\neq909=\vert
W_{\la_4}^{(1,2,1,2,1)}(3124)\vert.$$ 
See Table~\ref{table:2} for more
details.

\begin{table}[h]\tiny
\begin{tabular}{llllllllllllll}

\hline
\hline
${\mathbf a}$ & $(31111)$ & $(13111)$ & $(11311)$ & $(11131)$ & $(11113)$  \\
\hline
$\vert W_{\la_4}^{\mathbf a}(2314)\vert$ & $640$ & $640$ & $640$ & $635$ & $640$ \\
\hline
$\vert W_{\la_4}^{\mathbf a}(3124)\vert$ & $640$ & $640$ & $640$ & $635$ & $640$ \\
\hline
\\
\hline
\hline
${\mathbf a}$ & $(22111)$ & $(21211)$ & $(21121)$ & $(21112)$ & $(12211)$ & $(12121)$ &  $(12112)$ & $(11221)$ & $(11212)$ & $(11122)$ \\
\hline
$\vert W_{\la_4}^{\mathbf a}(2314)\vert$ & $913$ & $913$ & $909$ & $913$ & $913$ & $908$ & $913$ & $909$ & $913$ & $909$ \\
\hline
$\vert W_{\la_4}^{\mathbf a}(3124)\vert$ & $913$ & $913$ & $909$ & $913$ & $913$ & $909$ & $913$ & $909$ & $913$ & $909$ \\
\hline

\end{tabular}
\vskip10pt
\caption{}
\label{table:2}
\end{table}

Furthermore, we have 
$$\vert
W_{8,5}(2314)\vert=310540\neq310563=\vert W_{8,5}(3124)\vert.
$$
Table~\ref{table:3} shows that, with  $\la_5=(8,8,8,8,8)$, in fact
$$\vert W_{\la_5}^{(a_1,a_2,a_3,a_4,a_5)}(2314)\vert\neq\vert
W_{\la_5}^{(a_1,a_2,a_3,a_4,a_5)}(3124)\vert$$
for all possible sequences $(a_1,a_2,a_3,a_4,a_5)$ of positive integers. 

\begin{table}[h]\tiny
\begin{tabular}{llllllllllllll}

\hline
\hline
${\mathbf a}$ & $(13121)$ & $(12131)$ & $(11231)$ & $(22121)$ & $(12221)$ &$(12122)$ & $(11222)$  \\
\hline
$\vert W_{\la_5}^{\mathbf a}(2314)\vert$ & $2258$ & $2245$ & $2251$ & $3162$ & $3164$ & $3163$ & $3167$ \\
\hline
$\vert W_{\la_5}^{\mathbf a}(3124)\vert$ & $2263$ & $2251$ & $2250$ & $3167$ & $3167$ & $3167$ & $3168$ \\
\hline

\end{tabular}
\vskip10pt
\caption{}
\label{table:3}
\end{table}

For the case $\beta=12$, that is, for $23145$- respectively
$31245$-avoiding words ($0$-$1$-fillings), we found that
$$\vert W_{8,6}(23145)\vert=\vert
W_{8,6}(23154)\vert=1640298\neq1640299=\vert
W_{8,6}(31245)\vert=\vert
W_{8,6}(31254)\vert.$$ 
Here, the equalities follow from reversal and complementation
of patterns and Corollary~\ref{cor:7}.

\begin{table}[h]\tiny
\begin{tabular}{llllllllllllll}

\hline
\hline
$\la$ & $(6664)$ & $(88844)$ & $(99933)$ & $(99944)$ & $(775333)$ &  $(777744)$ & $(888664)$ & $(888844)$ & $(987654)$ \\
\hline
$\vert W_\la(231)\vert$ & $425$ & $4443$ & $6177$ & $13435$ & $70$ & $1012$ & $6232$ & $6160$ & $6183$  \\
\hline
$\vert W_\la(312)\vert$ & $429$ & $4443$ & $6177$ & $13435$ & $70$ & $1012$ & $6352$ & $6160$ & $6303$ \\
\hline
\\
\hline
\hline
$\la$ & $(987655)$ & $(987755)$ & $(996644)$ & $(997655)$ & $(997744)$ & $(997755)$ &  $(999663)$ & $(999755)$ & $(999944)$ \\
\hline
$\vert W_\la(231)\vert$ & $7301$ & $9133$ & $6130$ & $14602$ & $12870$ & $18266$ & $21549$ & $30517$ & $28036$ \\
\hline
$\vert W_\la(312)\vert$ & $7375$ & $9213$ & $6130$ & $14750$ & $12870$ & $18426$ & $21645$ & $31185$ & $28036$ \\
\hline

\end{tabular}
\vskip10pt
\caption{}
\label{table:4}
\end{table}

The data seem to suggest the following.

\begin{Conjecture} \label{conj:2}
For all non-empty permutations~$\beta$, the patterns
$231\oplus\beta$ and $312\oplus\beta$ are not Wilf-equivalent for words.
\end{Conjecture}

On the other hand,
based on more computational data (see for example Table~\ref{table:4}), 
we suspect that the following may be true.

\begin{Conjecture} \label{conj:1}
For every Ferrers shape $\la$, we have 
$\vert W_\la(231)\vert\leq\vert W_\la(312)\vert$,
where $W_\la(x)$ denotes the set of {\em all} $0$-$1$-fillings of
shape~$\la$ avoiding $x$ with exactly one~$1$ in each column.
\end{Conjecture}

\begin{Remark}
By going through the proof of Proposition~\ref{prop:5} in \cite{JeMaAA}, one sees that
the validity of Conjecture~\ref{conj:1} would imply that
$$\vert W_\la(231\oplus\beta)\vert\leq\vert W_\la(312\oplus\beta)\vert$$ 
for all patterns~$\beta$.
\end{Remark}

\section{``Modified" shape-Wilf-equivalence of $231$ and $312$ for words}
\label{sec:231-312}

We now show that a shape-Wilf-equivalence result can be obtained if we
do not insist on the ``strictness" of the patterns $231$ and $312$.
Namely, if in these patterns we include instances where there
is equality between the letters corresponding to the `$2$' and the
`$3$' in the patterns, or alternatively if we include instances
where there
is equality between the letters corresponding to the `$1$' and the
`$2$' in the patterns, then shape-Wilf-equivalence results do hold.
The precise formulations are given in the two theorems below.

\begin{Theorem} \label{lem:11} 
The sets of patterns $\{231,221\}$ and $\{312,212\}$ are
shape-Wilf-equivalent for words. 
\end{Theorem}

\begin{Theorem} \label{lem:12} 
The sets of patterns $\{231,121\}$ and $\{312,211\}$ are
shape-Wilf-equivalent for words. 
\end{Theorem}

Before we are able to prove these two theorems, we need to
recall the essential ingredients of the bijective proof of 
Theorem~\ref{thm:231-312} given by Bloom and Saracino in \cite{BlSaAA}.

Shape-Wilf-equivalence for permutations (as defined in
Section~\ref{sec:def}; see the paragraph after \eqref{eq:sW-equiv}) 
involves $0$-$1$-fillings with {\it exactly one~$1$ in each column},
and with {\it exactly one~$1$ in each row} (the latter coming from
the restriction $a_1=a_2=\dots=1$). We call such fillings {\it full rook
placements} from now on.

Consider a Ferrers shape $\la$ and a full rook placement $R$ on it.
For each vertex~$v$ (by which we mean a corner of a cell) 
along the right/up border of $\la$ we assign
an integer $I_R(v)$, which by definition is the length of the longest
increasing chain of $1$'s in the region to the left and below of~$v$.
The right of Figure~\ref{fig:3} shows the numbers $I_R(v)$ for the
particular full rook placement~$R$ presented there, as well does the left of
Figure~\ref{fig:4}. (At this point, the varying
thickness of lines should be ignored.)
Given a full rook placement $R$ on $\la$, we denote the sequence of
numbers $I_R(v)$, where $v$ ranges over the vertices along the right/up
border of $\la$ by $I(R)$.

Bloom and Saracino show the following:

\begin{enumerate} 
\item A full rook placement $R$ on~$\la$ that is $231$-avoiding
(as a $0$-$1$-filling) is uniquely determined by $I(R)$
(see \cite[Theorem~2]{BlSaAA}).
\item The possible sequences $I(R)$, where $R$ is a $231$-avoiding
full rook placement on~$\la$, have a simple characterisation,
by means of the so-called {\it $231$-conditions}
(see \cite{BlSaAA} for their definition).
\item A full rook placement $R$ on~$\la$ that is $312$-avoiding
(as a $0$-$1$-filling) is uniquely determined by $I(R)$
(see \cite[Theorem~2]{BlSaAA}).
\item The possible sequences $I(R)$, where $R$ is a $312$-avoiding
full rook placement on~$\la$, have a simple characterisation,
by means of the so-called {\it $312$-conditions}
(see \cite{BlSaAA} for their definition).
\item A bijection $\al$ from $231$- to $312$-avoiding full rook placements
of $\la$ can be defined as follows:
let $R_1$ be a $231$-avoiding full rook placement on~$\la$.
For each vertex~$v$ along the right/up border of $\la$ 
calculate the number 
\begin{equation} \label{eq:alpha} 
\begin{cases} 
0,&\text{if }I_{R_1}(v)=0,\\
N_{R_1}(v)-I_{R_1}(v)+1,&\text{otherwise},
\end{cases}
\end{equation}
where $N_{R}(v)$ denotes the
total number of $1$'s in a full rook placement~$R$ 
in the region to the left and below of~$v$.
This defines a new sequence of non-negative numbers.
Let $R_2$ be the uniquely determined $312$-avoiding full rook
placement corresponding to that sequence. The map $\al:R_1\to R_2$ is
a bijection.
\end{enumerate}

It should be noted that the full rook placements on the right of
Figure~\ref{fig:3} and on the left of Figure~\ref{fig:4} correspond to
each other under the bijection~$\al$.

\begin{proof}[Proof of Theorem~\ref{lem:11}] 
Let $\la$ be a Ferrers shape and $a_1,a_2,\dots$
a sequence of positive integers. 
We have to prove that
$\vert W_\la^{(a_1,a_2,\dots)}(231,221)\vert
=\vert W_\la^{(a_1,a_2,\dots)}(312,212)\vert$.
We are going to achieve this by constructing a bijection
between the sets
$W_\la^{(a_1,a_2,\dots)}(231,221)$ and
$W_\la^{(a_1,a_2,\dots)}(312,212)$. 

Let $T_1\in W_\la^{(a_1,a_2,\dots)}(231,221)$.
In general, the filling~$T_1$ is not a full rook placement since it may
contain several $1$'s in a row. However, we convert $T$ into a full
rook placement by replacing row $i$ by $a_i$ rows in the new filling,
$i=1,2,\dots$,
and rearranging the $a_i$ $1$'s in an increasing fashion from
bottom/left to top/right so that each of these~$1$'s stays in
its column but each of the $a_i$ rows contains
exactly one~$1$. For later reference, we call the region covered
by these $a_i$ rows the {\it $i$-th band}, and the obtained full
rook placement~$R_1$.
It should be noted that the original filling
is $\{231,221\}$-avoiding if and only if the new filling --- which
necessarily is a full rook placement --- is $231$-avoiding.
This construction is illustrated in
Figure~\ref{fig:3}. The left of the figure shows a filling
in $W^{(2,2,3,1,1,1)}_{(10,10,10,7,4,4)}(231,221)$.
The filling on the right shows the result of the above
described conversion.
In the figure, the separations between the original rows 
are indicated by thick lines, whereas the
newly created rows are separated by thin lines. The resulting filling
is in $W^{(1,1,\dots,1)}_{\la_1}(231)$, where 
$\la_1=(10,10,10,10,10,10,10,7,4,4)$.

\begin{figure}[h]
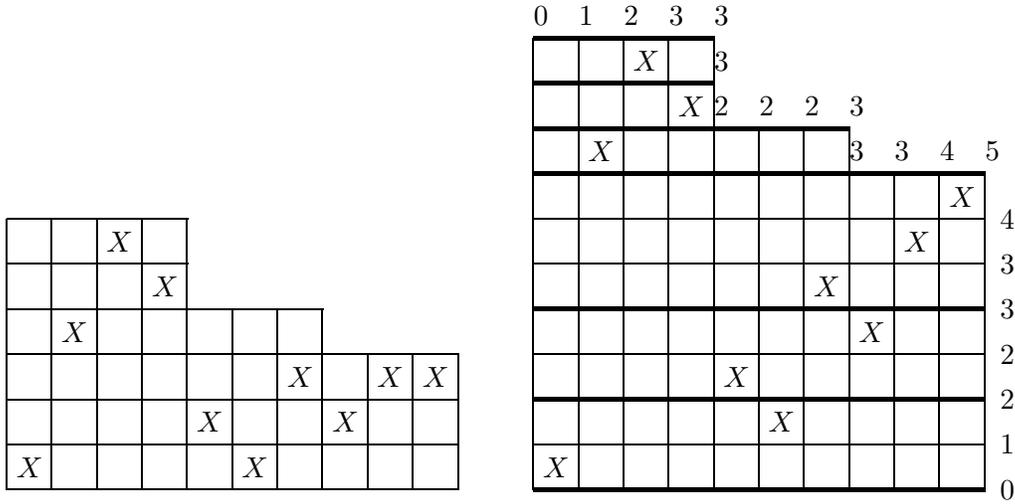

$$
\Einheit.2cm
\PfadDicke{.3pt}
\Pfad(0,18),111111111111\endPfad
\Pfad(0,15),111111111111\endPfad
\Pfad(0,12),111111111111111111111\endPfad
\Pfad(0,9),111111111111111111111111111111\endPfad
\Pfad(0,6),111111111111111111111111111111\endPfad
\Pfad(0,3),111111111111111111111111111111\endPfad
\Pfad(0,0),111111111111111111111111111111\endPfad
\Pfad(30,0),222222222\endPfad
\Pfad(27,0),222222222\endPfad
\Pfad(24,0),222222222\endPfad
\Pfad(21,0),222222222222\endPfad
\Pfad(18,0),222222222222\endPfad
\Pfad(15,0),222222222222\endPfad
\Pfad(12,0),222222222222222222\endPfad
\Pfad(9,0),222222222222222222\endPfad
\Pfad(6,0),222222222222222222\endPfad
\Pfad(3,0),222222222222222222\endPfad
\Pfad(0,0),222222222222222222\endPfad
\Label\ro{\text {\small$X$}}(1,1)
\Label\ro{\text {\small$X$}}(4,10)
\Label\ro{\text {\small$X$}}(7,16)
\Label\ro{\text {\small$X$}}(10,13)
\Label\ro{\text {\small$X$}}(13,4)
\Label\ro{\text {\small$X$}}(16,1)
\Label\ro{\text {\small$X$}}(19,7)
\Label\ro{\text {\small$X$}}(22,4)
\Label\ro{\text {\small$X$}}(25,7)
\Label\ro{\text {\small$X$}}(28,7)
\hbox{\hskip7cm}
\PfadDicke{.3pt}
\Pfad(30,0),222222222222222222222\endPfad
\Pfad(27,0),222222222222222222222\endPfad
\Pfad(24,0),222222222222222222222\endPfad
\Pfad(21,0),222222222222222222222222\endPfad
\Pfad(18,0),222222222222222222222222\endPfad
\Pfad(15,0),222222222222222222222222\endPfad
\Pfad(12,0),222222222222222222222222222222\endPfad
\Pfad(9,0),222222222222222222222222222222\endPfad
\Pfad(6,0),222222222222222222222222222222\endPfad
\Pfad(3,0),222222222222222222222222222222\endPfad
\Pfad(0,0),222222222222222222222222222222\endPfad
\PfadDicke{1.5pt}
\Pfad(0,30),111111111111\endPfad
\Pfad(0,27),111111111111\endPfad
\Pfad(0,24),111111111111111111111\endPfad
\Pfad(0,21),111111111111111111111111111111\endPfad
\PfadDicke{.3pt}
\Pfad(0,18),111111111111111111111111111111\endPfad
\Pfad(0,15),111111111111111111111111111111\endPfad
\PfadDicke{1.5pt}
\Pfad(0,12),111111111111111111111111111111\endPfad
\PfadDicke{.3pt}
\Pfad(0,9),111111111111111111111111111111\endPfad
\PfadDicke{1.5pt}
\Pfad(0,6),111111111111111111111111111111\endPfad
\PfadDicke{.3pt}
\Pfad(0,3),111111111111111111111111111111\endPfad
\PfadDicke{1.5pt}
\Pfad(0,0),111111111111111111111111111111\endPfad
\Label\ro{\text {\small$X$}}(1,1)
\Label\ro{\text {\small$X$}}(4,22)
\Label\ro{\text {\small$X$}}(7,28)
\Label\ro{\text {\small$X$}}(10,25)
\Label\ro{\text {\small$X$}}(13,7)
\Label\ro{\text {\small$X$}}(16,4)
\Label\ro{\text {\small$X$}}(19,13)
\Label\ro{\text {\small$X$}}(22,10)
\Label\ro{\text {\small$X$}}(25,16)
\Label\ro{\text {\small$X$}}(28,19)
\Label\ro{\text {\small$0$}}(0,31)
\Label\ro{\text {\small$1$}}(3,31)
\Label\ro{\text {\small$2$}}(6,31)
\Label\ro{\text {\small$3$}}(9,31)
\Label\ro{\text {\small$3$}}(12,31)
\Label\ro{\text {\small$3$}}(12,28)
\Label\ro{\text {\small$2$}}(12,25)
\Label\ro{\text {\small$2$}}(15,25)
\Label\ro{\text {\small$2$}}(18,25)
\Label\ro{\text {\small$3$}}(21,25)
\Label\ro{\text {\small$3$}}(21,22)
\Label\ro{\text {\small$3$}}(24,22)
\Label\ro{\text {\small$4$}}(27,22)
\Label\ro{\text {\small$5$}}(30,22)
\Label\r{\text {\small$4$}}(30,18)
\Label\r{\text {\small$3$}}(30,15)
\Label\r{\text {\small$3$}}(30,12)
\Label\r{\text {\small$2$}}(30,9)
\Label\r{\text {\small$2$}}(30,6)
\Label\r{\text {\small$1$}}(30,3)
\Label\r{\text {\small$0$}}(30,0)
\hskip5.8cm
$$
\caption{The blowup of a filling}
\label{fig:3}
\end{figure}

We now apply the bijection~$\al$ to $R_1$, to obtain a 
$312$-avoiding full rook
placement~$R_2.$
The left of Figure~\ref{fig:4} shows the result when we apply $\al$
to the full rook placement on the right of Figure~\ref{fig:3}.

\smallskip
{\sc Claim 1.}
{\em In the $i$-th band, the $1$'s in
$R_2$ are arranged in a decreasing fashion, from top/left to
bottom/right, for each~$i$.}
\smallskip

Finally, we shrink the $i$-th band back to a single row, putting a~$1$
in those columns where $R_2$ contained a~$1$. We denote the obtained
filling by~$T_2$. Since $R_2$ was $312$-avoiding, the ``compressed"
filling will be $\{312,212\}$-avoiding --- if we take the above claim
for granted. Thus, the filling~$T_2$ is in 
$W_\la^{(a_1,a_2,\dots)}(312,212)$.
Figure~\ref{fig:4} shows the above shrinking process applied to our
running example.

From the construction, it is also obvious how the inverse mapping
works. It involves however another claim.

\smallskip
{\sc Claim 2.}
{\em If the inverse of $\al$
is applied to a full rook placement with the property that, in
the $i$-th band, the $1$'s are arranged in decreasing fashion,
then we obtain a full rook placement where, in
the $i$-th band, the $1$'s are arranged in increasing fashion,
$i=1,2,\dots$.}
\smallskip

If we assume the truth of these two claims then the proof of
Theorem~\ref{lem:11} is complete.

\begin{figure}[h]
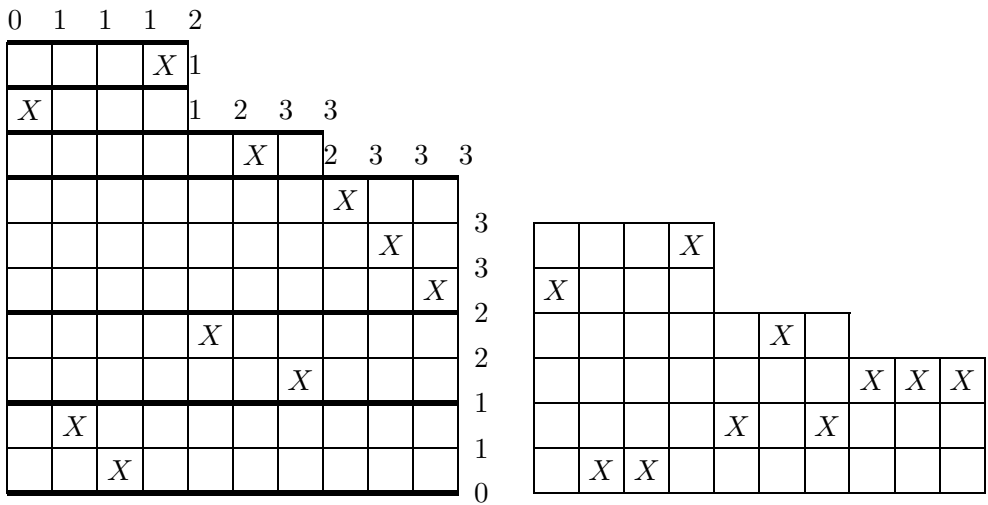

$$
\Einheit.2cm
\PfadDicke{.3pt}
\Pfad(30,0),222222222222222222222\endPfad
\Pfad(27,0),222222222222222222222\endPfad
\Pfad(24,0),222222222222222222222\endPfad
\Pfad(21,0),222222222222222222222222\endPfad
\Pfad(18,0),222222222222222222222222\endPfad
\Pfad(15,0),222222222222222222222222\endPfad
\Pfad(12,0),222222222222222222222222222222\endPfad
\Pfad(9,0),222222222222222222222222222222\endPfad
\Pfad(6,0),222222222222222222222222222222\endPfad
\Pfad(3,0),222222222222222222222222222222\endPfad
\Pfad(0,0),222222222222222222222222222222\endPfad
\PfadDicke{1.5pt}
\Pfad(0,30),111111111111\endPfad
\Pfad(0,27),111111111111\endPfad
\Pfad(0,24),111111111111111111111\endPfad
\Pfad(0,21),111111111111111111111111111111\endPfad
\PfadDicke{.3pt}
\Pfad(0,18),111111111111111111111111111111\endPfad
\Pfad(0,15),111111111111111111111111111111\endPfad
\PfadDicke{1.5pt}
\Pfad(0,12),111111111111111111111111111111\endPfad
\PfadDicke{.3pt}
\Pfad(0,9),111111111111111111111111111111\endPfad
\PfadDicke{1.5pt}
\Pfad(0,6),111111111111111111111111111111\endPfad
\PfadDicke{.3pt}
\Pfad(0,3),111111111111111111111111111111\endPfad
\PfadDicke{1.5pt}
\Pfad(0,0),111111111111111111111111111111\endPfad
\Label\ro{\text {\small$X$}}(1,25)
\Label\ro{\text {\small$X$}}(4,4)
\Label\ro{\text {\small$X$}}(7,1)
\Label\ro{\text {\small$X$}}(10,28)
\Label\ro{\text {\small$X$}}(13,10)
\Label\ro{\text {\small$X$}}(16,22)
\Label\ro{\text {\small$X$}}(19,7)
\Label\ro{\text {\small$X$}}(22,19)
\Label\ro{\text {\small$X$}}(25,16)
\Label\ro{\text {\small$X$}}(28,13)
\Label\ro{\text {\small$0$}}(0,31)
\Label\ro{\text {\small$1$}}(3,31)
\Label\ro{\text {\small$1$}}(6,31)
\Label\ro{\text {\small$1$}}(9,31)
\Label\ro{\text {\small$2$}}(12,31)
\Label\ro{\text {\small$1$}}(12,28)
\Label\ro{\text {\small$1$}}(12,25)
\Label\ro{\text {\small$2$}}(15,25)
\Label\ro{\text {\small$3$}}(18,25)
\Label\ro{\text {\small$3$}}(21,25)
\Label\ro{\text {\small$2$}}(21,22)
\Label\ro{\text {\small$3$}}(24,22)
\Label\ro{\text {\small$3$}}(27,22)
\Label\ro{\text {\small$3$}}(30,22)
\Label\r{\text {\small$3$}}(30,18)
\Label\r{\text {\small$3$}}(30,15)
\Label\r{\text {\small$2$}}(30,12)
\Label\r{\text {\small$2$}}(30,9)
\Label\r{\text {\small$1$}}(30,6)
\Label\r{\text {\small$1$}}(30,3)
\Label\r{\text {\small$0$}}(30,0)
\hbox{\hskip7cm}
\PfadDicke{.3pt}
\Pfad(0,18),111111111111\endPfad
\Pfad(0,15),111111111111\endPfad
\Pfad(0,12),111111111111111111111\endPfad
\Pfad(0,9),111111111111111111111111111111\endPfad
\Pfad(0,6),111111111111111111111111111111\endPfad
\Pfad(0,3),111111111111111111111111111111\endPfad
\Pfad(0,0),111111111111111111111111111111\endPfad
\Pfad(30,0),222222222\endPfad
\Pfad(27,0),222222222\endPfad
\Pfad(24,0),222222222\endPfad
\Pfad(21,0),222222222222\endPfad
\Pfad(18,0),222222222222\endPfad
\Pfad(15,0),222222222222\endPfad
\Pfad(12,0),222222222222222222\endPfad
\Pfad(9,0),222222222222222222\endPfad
\Pfad(6,0),222222222222222222\endPfad
\Pfad(3,0),222222222222222222\endPfad
\Pfad(0,0),222222222222222222\endPfad
\Label\ro{\text {\small$X$}}(1,13)
\Label\ro{\text {\small$X$}}(4,1)
\Label\ro{\text {\small$X$}}(7,1)
\Label\ro{\text {\small$X$}}(10,16)
\Label\ro{\text {\small$X$}}(13,4)
\Label\ro{\text {\small$X$}}(16,10)
\Label\ro{\text {\small$X$}}(19,4)
\Label\ro{\text {\small$X$}}(22,7)
\Label\ro{\text {\small$X$}}(25,7)
\Label\ro{\text {\small$X$}}(28,7)
\hskip6cm
$$
\caption{The shrinking of a full rook placement}
\label{fig:4}
\end{figure}

\medskip
We now prove Claim~1. We want to establish that, if~$R_1$
is a $231$-avoiding full rook placement with
$1$'s arranged in increasing fashion in the $i$-th band, then
$R_2=\al(R_1)$ has $1$'s arranged in decreasing fashion in the
$i$-th band, $i=1,2,\dots$. Let us concentrate on the $i$-th band 
of $R_1$. We denote the vertices along the right border of the
band by $u_0,u_1,\dots,u_{a_i}$. 
In order to follow the next arguments, it might be helpful to look
at Figure~\ref{fig:5}. The left of Figure~\ref{fig:5} is meant to be
the sketch of the $i$-th band (with $a_i=4$). The $1$'s in the full
rook placement $R_1$ are indicated by $X_1,X_2,\dots$.

\begin{figure}[h]
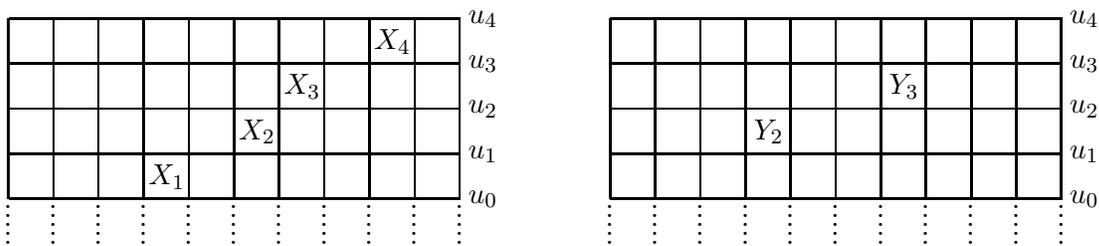

$$
\Einheit.2cm
\PfadDicke{.5pt}
\Pfad(0,12),111111111111111111111111111111\endPfad
\Pfad(0,9),111111111111111111111111111111\endPfad
\Pfad(0,6),111111111111111111111111111111\endPfad
\Pfad(0,3),111111111111111111111111111111\endPfad
\Pfad(0,0),111111111111111111111111111111\endPfad
\Pfad(30,0),222222222222\endPfad
\Pfad(27,0),222222222222\endPfad
\Pfad(24,0),222222222222\endPfad
\Pfad(21,0),222222222222\endPfad
\Pfad(18,0),222222222222\endPfad
\Pfad(15,0),222222222222\endPfad
\Pfad(12,0),222222222222\endPfad
\Pfad(9,0),222222222222\endPfad
\Pfad(6,0),222222222222\endPfad
\Pfad(3,0),222222222222\endPfad
\Pfad(0,0),222222222222\endPfad
\Label\ro{\text {\small$X_1$}}(10,1)
\Label\ro{\text {\small$X_2$}}(16,4)
\Label\ro{\text {\small$X_3$}}(19,7)
\Label\ro{\text {\small$X_4$}}(25,10)
\Label\r{\text {\small$u_0$}}(30,0)
\Label\r{\text {\small$u_1$}}(30,3)
\Label\r{\text {\small$u_2$}}(30,6)
\Label\r{\text {\small$u_3$}}(30,9)
\Label\r{\text {\small$u_4$}}(30,12)
\Einheit.6cm
\SPfad(0,-1),2\endSPfad
\SPfad(1,-1),2\endSPfad
\SPfad(2,-1),2\endSPfad
\SPfad(3,-1),2\endSPfad
\SPfad(4,-1),2\endSPfad
\SPfad(5,-1),2\endSPfad
\SPfad(6,-1),2\endSPfad
\SPfad(7,-1),2\endSPfad
\SPfad(8,-1),2\endSPfad
\SPfad(9,-1),2\endSPfad
\SPfad(10,-1),2\endSPfad
\hbox{\hskip8cm}
\Einheit.2cm
\PfadDicke{.5pt}
\Pfad(0,12),111111111111111111111111111111\endPfad
\Pfad(0,9),111111111111111111111111111111\endPfad
\Pfad(0,6),111111111111111111111111111111\endPfad
\Pfad(0,3),111111111111111111111111111111\endPfad
\Pfad(0,0),111111111111111111111111111111\endPfad
\Pfad(30,0),222222222222\endPfad
\Pfad(27,0),222222222222\endPfad
\Pfad(24,0),222222222222\endPfad
\Pfad(21,0),222222222222\endPfad
\Pfad(18,0),222222222222\endPfad
\Pfad(15,0),222222222222\endPfad
\Pfad(12,0),222222222222\endPfad
\Pfad(9,0),222222222222\endPfad
\Pfad(6,0),222222222222\endPfad
\Pfad(3,0),222222222222\endPfad
\Pfad(0,0),222222222222\endPfad
\Label\ro{\text {\small$Y_2$}}(10,4)
\Label\ro{\text {\small$Y_3$}}(19,7)
\Label\r{\text {\small$u_0$}}(30,0)
\Label\r{\text {\small$u_1$}}(30,3)
\Label\r{\text {\small$u_2$}}(30,6)
\Label\r{\text {\small$u_3$}}(30,9)
\Label\r{\text {\small$u_4$}}(30,12)
\Einheit.6cm
\SPfad(0,-1),2\endSPfad
\SPfad(1,-1),2\endSPfad
\SPfad(2,-1),2\endSPfad
\SPfad(3,-1),2\endSPfad
\SPfad(4,-1),2\endSPfad
\SPfad(5,-1),2\endSPfad
\SPfad(6,-1),2\endSPfad
\SPfad(7,-1),2\endSPfad
\SPfad(8,-1),2\endSPfad
\SPfad(9,-1),2\endSPfad
\SPfad(10,-1),2\endSPfad
\hskip6cm
$$
\caption{The $i$-th band}
\label{fig:5}
\end{figure}
 
We claim that
\begin{equation} \label{eq:IR1} 
I_{R_1}(u_{j+1})=I_{R_1}(u_j)+1,\quad  \text{for }j=1,2,\dots,a_i-1.
\end{equation}
Indeed, by definition, $I_{R_1}(u_1)$ is the length of the
longest increasing chain of $1$'s in the region to the left and below
of~$u_1$. Since there cannot be any $1$'s in the region to the right
and below of~$X_2$ because $R_1$ is $231$-avoiding, we must have
$I_{R_1}(u_2)=I_{R_1}(u_1)+1$. The same argument shows
$I_{R_1}(u_3)=I_{R_1}(u_2)+1$, etc., thus establishing our claim in
\eqref{eq:IR1}.

In order to apply the bijection~$\al$, we have to calculate the
numbers in \eqref{eq:alpha}, which then become lengths of increasing
chains $I_{R_2}(v)$ for some $312$-avoiding full rook placement~$R_2$. 
Because of \eqref{eq:IR1}, we have 
\begin{equation} \label{eq:R1R2}
I_{R_2}(u_{j+1})=I_{R_2}(u_j),\quad  \text{for }j=1,2,\dots,a_i-1. 
\end{equation}
Let us suppose that $R_2$ has two $1$'s in the $i$-th band which are
in increasing order. Without loss of generality, we may assume that
these two $1$'s are located in neighbouring rows.
The right of Figure~\ref{fig:5} is meant to
illustrate this situation, with $Y_2$ and $Y_3$ indicating the
positions of these two $1$'s. Since $R_2$ is $312$-avoiding,
$1$'s in the region to
the right and below of $Y_2$ must be arranged in decreasing order.
Thus, together with $Y_2$, they form a decreasing chain.
The number $I_{R_2}(u_2)$ gives the length of the longest increasing
chain in the region to the left and below of~$u_2$. Because of the
above observed arrangement of $1$'s to the right and below of $Y_2$, 
one of the chains with length $I_{R_2}(u_2)$ is one which ends 
in~$Y_2$. However, this implies that $I_{R_2}(u_3)=I_{R_2}(u_2)+1$,
in contradiction with \eqref{eq:R1R2}. Thus, $Y_2$ and $Y_3$ cannot
be arranged in increasing fashion. The same argument applies to
any two neighbouring rows in the $i$-th band. Thus, we have
proven Claim~1.

\medskip
The only remaining task for the completion of the proof
of the theorem is to verify Claim~2. This is however
completely analogous and is left to the reader.
\end{proof} 

\begin{proof}[Proof of Theorem~\ref{lem:12}] 
This can be proven in a fashion which is analogous to the proof of
  Theorem~\ref{lem:11}. The major difference is that, here,
a $\{231,121\}$-avoiding filling is converted into a $231$-avoiding
full rook placement by rearranging the $a_i$ $1$'s in the $i$-th band
in {\it decreasing} fashion, while a 
$\{312,211\}$-avoiding filling is converted into a $312$-avoiding
full rook placement by rearranging the $a_i$ $1$'s in the $i$-th band
in {\it increasing} fashion. We leave the details to the reader.
\end{proof}

In view of Proposition~\ref{prop:5} and the fact that it can easily 
be extended to {\it sets} of patterns, Theorems~\ref{lem:11}
and~\ref{lem:12} entail the following corollaries.

\begin{Corollary} \label{cor:11} 
For any word $\beta$,
the sets of patterns $\{231\oplus\beta,221\oplus\beta\}$ 
and $\{312\oplus\beta,212\oplus\beta\}$ are
shape-Wilf-equivalent for words. 
\end{Corollary}

\begin{Corollary} \label{cor:12} 
For any word $\beta$,
the sets of patterns $\{231\oplus\beta,121\oplus\beta\}$ 
and $\{312\oplus\beta,211\oplus\beta\}$ are
shape-Wilf-equivalent for words. 
\end{Corollary}

\begin{Remark}
It is interesting to note that $\{312,212\}$-avoiding words
appear in a ``Catalan" context in \cite{CeDLAA}. 
Let $\mathbf s=(s_1,s_2,\dots,s_a)$ be an $a$-tuple of positive integers.
In the above paper, Gessel and
Stanley's \cite{GeStAA} Stirling permutations are generalised
there to Stirling $\mathbf s$-(multi)permutations, which by definition 
are --- in our language here --- $212$-avoiding words 
consisting of $s_1$ letters~1, $s_2$ letters~2, \dots, and
$s_a$ letters~$a$. It is shown (see \cite[Theorem~2.1]{CeDLAA})
that $312$-avoiding Stirling $\mathbf s$-permutations are in bijection
with numerous other combinatorial objects subject to certain
restrictions, including planted plane trees, lattice paths,
non-crossing partitions, and dissections of polygons.
\end{Remark}

\section*{Acknowledgments}
The authors thank Doron Zeilberger for raising the problems 
addressed in this paper, and Sergey Kitaev for very helpful
correspondence.

\end{document}